\documentclass[letterpaper,12pt]{article}
\usepackage{amssymb,amsthm,graphicx,multicol,amsmath,latexsym,amsmath,amssymb,ebgaramond,fancyhdr, cancel,enumitem,mathrsfs,ulsy}
\newtheorem{theorem}{Theorem}[section]
\newtheorem{lemma}[theorem]{Lemma}
\newtheorem{corollary}[theorem]{Corollary}
\newtheorem{proposition}[theorem]{Proposition}

\newtheorem{assumption}[theorem]{Assumption}

\def\N{\mathbb{N}}
\def\Z{\mathbb{Z}}

\def\into{\longrightarrow}

\def\sub{\subseteq}

\def\lbrak{\left[}
\def\rbrak{\right]}
\def\lpar{\left(}
\def\rpar{\right)}

\def\setdiff{\backslash}

\def\N{{\mathbb N}}
\def\Z{{\mathbb Z}}

\def\F{\mathbb{F}}
\def\scA{\mathcal{A}}

\def\scK{\mathcal{K}}
\def\scKii{\widetilde{\mathcal{K}}}
\def\scI{\mathcal {I}}

\def\genset{\mathcal{X}}

\def\scY{\mathcal{Y}}

\def\scS{\mathcal{S}}

\def\A{\alpha}

\def\z{\zeta}

\def\L{\lambda}

\def\TT{\tau}
\def\S{\sigma}

\def\P{\Phi}
\def\x{\xi}
\def\m{\mu}

\def\ad{{\rm ad}\ }

\def\U{\mathcal{U}_q(r, s)}
\def\Ur{\mathcal{U}_q(r, 0)}
\def\UrOne{\mathcal{U}_q(1, 0)}
\def\Us{\mathcal{U}_q(0, s)}
\def\UsOne{\mathcal{U}_q(0, 1)}
\def\Uss{\mathcal{U}_q(s,0)}

\def\genset{\mathcal{X}}
\def\monoid{\left<\genset\right>}

\def\freeF{\F\monoid}
\def\freeFtwo{\F\left<\genset_2\right>}
\def\freeY{\F\left<\scY\right>}
\def\Liesub{\mathcal{L}}
\def\LiesubUr{\mathcal{L}_{(r,0)}}
\def\LiesubUss{\mathcal{L}_{(s,0)}}
\def\LiesubUs{\mathcal{L}_{(0,s)}}

\def\r{\mathfrak{r}}
\def\a{\mathfrak{a}}
\def\b{\mathfrak{b}}
\def\LieIso{\Psi}
\def\nonHomo{\Phi}

\def\univ{\mathcal{U}}
\def\theLie{\mathfrak{g}}
\def\LieA{\theLie_r}
\def\LieB{\widetilde{\theLie}_s}
\def\uLieA{\univ(\LieA)}
\def\uLieB{\univ(\LieB)}
\def\LieAone{\theLie_1}
\def\LieBone{\widetilde{\theLie}_1}
\def\uLieAone{\univ(\LieAone)}
\def\uLieBone{\univ(\LieBone)}

\DeclareMathOperator{\Span}{Span}

\makeatletter
\newcommand\footnoteref[1]{\protected@xdef\@thefnmark{\ref{#1}}\@footnotemark}
\makeatother

\begin{document}

\hrule \vspace{3pt}\medskip
\noindent{\small {\bf Lie polynomials in a $q$-deformed universal enveloping algebra of a low-dimensional {L}ie algebra} \hfill\\ \emph{Rafael Reno S. Cantuba}\footnote{\label{dms}Department of Mathematics and Statistics,  De La Salle University, 2401 Taft Ave., Malate, Manila, Philippines}$^{,}$\footnote{Corresponding author: Email: rafael.cantuba@dlsu.edu.ph, ORCiD: 0000-0002-4685-8761}, \emph{Mark Anthony C. Merciales}\footnoteref{dms}$^{,}$\footnote{The second author acknowledges the support of the Science Education Institute of the Department of Science and Technology (DOST-SEI), Republic of the Philippines.}} \medskip\hrule\ 

\begin{quote}  \textsc{{Abstract}}. The nonabelian two-dimensional Lie algebra over a field $\mathbb{F}$ has a presentation by generators $A$, $B$ and relation $\left[ A,B\right]=A$, with the universal enveloping algebra having a presentation by generators $A$, $B$ and relation $AB-BA=A$. A well-known fact is that the said Lie algebra is isomorphic to that which has a universal enveloping algebra that has a presentation by generators $A$, $B$ and relation $AB-BA=B$. Given $q,r,s\in\mathbb{F}$, solutions to the Lie polynomial characterization problems in the corresponding $q$-deformed universal enveloping algebras, with generalized relations $AB-qBA=rA$ and $AB-qBA=sB$, respectively, are presented.\\


\end{quote}

\section{\normalsize Introduction}

The nonabelian two-dimensional Lie algebra $\LieAone$, over a field $\F$, is unique up to isomorphism. See, for instance, \cite[Theorem 3.1]{erd06}. A basis for $\LieAone$ consists of $A$ and $B$ that satisfy the relation $\lbrak A,B\rbrak=A$. The universal enveloping algebra of $\LieAone$ is the (associative) algebra $\uLieAone$ that has a presentation by generators $A$, $B$ and relation $AB-BA=A$. Let $q\in\F$. In this work, we shall be interested in what is called the ``$q$-analog'' or ``$q$-deformation'' of the Lie bracket operation that was done on $A$ and $B$, which results to the expression $AB-qBA$. The literature on $q$-analogs or $q$-deformations is extensive. With reference to the scope of this paper, what shall suffice is to mention here two of the important achievements made using $q$-analogs. The study of $q$-analogs of notions from ordinary calculus led to the discovery of many important notions and results in combinatorics, number theory, and other fields of mathematics \cite[p. vii]{kac02}, and the $q$-analogs of commutation relations of important Hilbert space operators have been successfully applied to, for instance, particle physics, knot theory and general relativity \cite[Chapter 12]{ern12}.

Since the $q$-deformation of the Lie bracket shall be considered later, we now mention some of the isomorphic forms of $\LieAone$ so that there shall be more clarity as to which of the isomorphisms are carried over, or not, after the $q$-deformation. If $\LieBone$ is the Lie algebra over $\F$ with a basis consisting of $A$ and $B$, subject to the relation $\lbrak A,B\rbrak=B$, then there exists a Lie algebra isomorphism $\LieAone\into\LieBone$ such that $A\mapsto B$ and $B\mapsto -A$. Given a nonzero $r\in\F$, if $\LieA$ is the Lie algebra over $\F$ with a basis consisting of $A$, $B$ that satisfy the commutation relation $\lbrak A,B\rbrak=rA$, then there exists a Lie algebra isomorphism $\LieA\into\LieAone$ that sends $A\mapsto A$ and $B\mapsto rB$. Also of interest here is the Lie algebra $\LieB$ over $\F$ with a basis consisting of $A$ and $B$ that obey the relation $\lbrak A,B\rbrak=sB$, given a nonzero $s\in\F$. There exists a Lie algebra isomorphism $\LieB\into\LieBone$ with the property that $A\mapsto sA$ and $B\mapsto B$. The universal enveloping algebras $\uLieA$, $\uLieB$, for all $r,s\in\F\setdiff\{0\}$, of the aforementioned Lie algebras are isomorphic. This may be proven routinely by the universal property of these universal enveloping algebras, together with the fact that, for the said isomorphic forms of the nonabelian two-dimensional Lie algebra, the inclusion maps $\LieA\hookrightarrow\uLieA$ and $\LieB\hookrightarrow\uLieB$ are injective \cite[Exercise 17.2]{hum72}. The isomorphism previously mentioned for the Lie algebras have their natural extension to algebra isomorphisms of the universal enveloping algebras. As a recapitulation, these algebra isomorphisms have the properties
\begin{eqnarray}
\uLieAone\into\uLieBone &: & A\mapsto B, \quad B\mapsto -A,\label{notMap1}\\
\uLieA\into\uLieAone &: &  A\mapsto A, \quad B\mapsto rB,\label{notMap2}\\
\uLieB\into\uLieBone &: & A\mapsto sA, \quad B\mapsto B.\label{notMap3}
\end{eqnarray}
Given a nonzero $q\in\F$, we now consider the $q$-deformation of the Lie bracket in the above algebras. Given a nonzero $r\in\F$, the algebra $\uLieA$ has a presentation by generators $A$, $B$ and relation $AB-BA=rA$. The corresponding $q$-deformed algebra is what we shall denote by $\Ur$ that has a presentation by generators $A$, $B$ and relation $AB-qBA=rA$. Analogously, the $q$-deformation $\Us$ for $\uLieB$ has a presentation by generators $A$, $B$ and relation $AB-qBA=sB$, where $s\in\F$ is nonzero.

The algebra $\uLieAone$ has a natural Lie algebra stucture induced by the operation that sends any\linebreak $X,Y\in\uLieAone$ to $XY-YX$. The Lie subalgebra of $\uLieAone$ generated by $A$ and $B$ is simply $\LieAone$, because of the relation $AB-BA=A$, and this is at the heart of the theory of universal enveloping algebras. This reduction of the Lie subalgebra to a smaller substructure is not necessarily true anymore for the algebra $\UrOne$. In $\UrOne$, we may still compute for Lie algebra expressions generated by $A$ and $B$, but the new relation $AB-qBA=A$ does not imply that the Lie subalgebra of $\UrOne$ generated by $A$ and $B$, or the set of all ``Lie polynomials'' in $A,B\in\UrOne$, shall be reduced into a small substructure. This is the main goal of this paper: to characterize all the Lie polynomials in $A$ and $B$ under a relation like $AB-qBA=A$, or what can be called the ``Lie polynomial characterization problem'' for the given presentation of $\UrOne$ by generators and relations. In particular, we want to solve this problem for the generalized algebras $\Ur$ and $\Us$.

Lie polynomial characterization problems were first studied in \cite{can15}, which was about the universal Askey-Wilson algebra, an important mathematical object in algebraic combinatorics, which arose from mathematical physics. The Lie polynomial characterization problem was completely solved for the $q$-deformed Heisenberg algebra and some extensions of this algebraic structure \cite{can19,can20,can21,can22,canmer21,cansil20a,cansil20b}. The $q$-deformed Heisenberg algebra \cite{hel00,hel05} is a $q$-analog of the Heisenberg-Weyl algebra \cite{can24}, which is an algebraic structure important in quantum theory.

Another central theme of this paper is that, when $q\neq 1$, there are no algebra homomorphisms that correspond to \eqref{notMap1}--\eqref{notMap3}, for the respective $q$-deformed algebras. We prove this in Corollaries~\ref{rCor}, \ref{sCor} and \ref{noIsoCor}. However, given some conditions on $r$ and $s$, the Lie algebra $\LiesubUr$ of all Lie polynomials in $A,B\in\Ur$ is isomorphic to the corresponding Lie subalgebra $\LiesubUs$ of $\Us$. After solving the Lie polynomial characterization problem in $\Ur$ in the proof of Theorem~\ref{TheThm}, we show in Theorem~\ref{TheIsoThm} that one Lie algebra isomorphism $\LiesubUr\into\LiesubUs$, under some conditions on $r$ and $s$, performs $A\mapsto -B$ and $B\mapsto -A$.

The aforementioned results were obtained with the aid of the Diamond Lemma for Ring Theory \cite[Theorem 1.2]{ber77}, which is an ingenious and indispensable tool in the determination of a basis for an algebra given a certain kind of presentation. The proofs and computations for the algebras $\Ur$ and $\Us$ that are based on the Diamond Lemma are analogous when done separately. For a better presentation of these proofs and computations, we decided to generalize the algebras $\Ur$ and $\Us$ into an algebra $\U$ which has a presentation by generators $A$, $B$ and relation $AB-qBA=rA+sB$. The basis theorem, Theorem~\ref{theobasisfor3gen}, is valid not only for the aforementioned restrictions on $r$ and $s$ for the relevant Lie algebras, but also for any choice of $r$ and $s$ in the field $\F$.

\section{\normalsize Preliminaries}

Given a field $\F$, any $\F$-algebra shall be assumed to be associative and unital. Since only one field $\F$ will be used, we further drop the prefix ``$\F$-'' and simply use the term algebra. A Lie algebra structure is induced on an algebra $\scA$ by the operation $[X, Y]:=XY-YX$ for all $X, Y \in \scA$. If $A_1, A_2, \ldots, A_n \in \scA$, then the Lie subalgebra $\scK$ of $\scA$ generated  by $A_1, A_2, \ldots, A_n$ is the smallest Lie subalgebra which contains  $A_1, A_2, \ldots, A_n$. That is, if $\scS$ is a Lie subalgebra of $\scA$, if $\{A_1, A_2, \ldots, A_n\} \subseteq \scS$, and if $\scS \subseteq \scK$, then $\scS=\scK$. In such a case, we refer to the elements of $\scK$ as \emph{Lie polynomials in $A_1, A_2, \ldots, A_n$}.

We denote the set of all nonnegative integers by $\N$, and the set of all positive integers by $\Z^+$. We fix $\nu\in \Z^+$, and let $\genset=\genset_\nu$ be a set with $\nu$ elements. The free monoid on $\genset$ shall be denoted by $\monoid$, while the free algebra generated by $\genset$ shall be denoted by $\freeF$. Most of the fundamental notions and properties of the aforementioned free monoid and free algebra may be seen, for instance, from \cite[Chapter~ 1]{lot97} or \cite[Section~1.1]{reu03}, and we proceed with only the terminology and notation necessary. Any basis element of $\freeF$ from the basis $\monoid$ is called a \emph{word} on $\genset$. The \emph{length} of a word $W\in \monoid$ shall be  denoted by $|W|$. Words of length $1$ are precisely the elements of $\genset$, and are called \emph{letters}. The word of length $0$ is called the \emph{empty word} in $\monoid$ and shall denoted by $I$,  which is also the identity element under the concatenation multiplication in $\monoid$. If $|W|\neq 0$, then $W$ is said to be a \emph{nonempty word}, but we further define $W^0:=I$. Multiplication in $\freeF$ is determined by the concatenation product in $\monoid$. Given a word $W\in \monoid$, a word $W'$ is said to be a \emph{subword} of $W$ if there exist words $L, R \in  \monoid$ such that $W=LW'R$.

If $\genset=\{X_1, \, X_2, \, \ldots, \,  X_\nu\}$, given $L_1$, $R_1$, $L_2$, $R_2$, $\ldots$, $L_m$, $R_m \in \freeF$, let $\scI$ be the (two-sided) ideal of $\freeF$ generated by $L_1 - R_1$, $L_2 - R_2$, $\ldots$ , $L_m - R_m$.  The \emph{algebra with generators $X_1  , X_2 , \ldots , X_n$ and relations $L_1 = R_1$, $L_2 = R_2$, $\ldots$ , $L_m = R_m$}
is the quotient algebra $\freeF/\scI$. With respect to the natural embedding $\genset\hookrightarrow\freeF/\scI$, if $\scK$ is the Lie subalgebra of $\freeF/\scI$ generated by $X_1, \, X_2, \, \ldots, \,  X_\nu$, then a characterization of the elements of $\scK$ is said to be a solution to the \emph{Lie polynomial characterization problem} with respect to the aforementioned presentation of $\freeF/\scI$.

We recall \emph{Bergman's Diamond Lemma} or the \emph{Diamond Lemma for Ring Theory} \cite[Theorem 1.2]{ber77}, together with some related notions taken from \cite[Section 1]{ber77}, which are crucial in determining a basis for $\freeF/\scI$. A set of ordered pairs of the form $\L=(W_\L , f_\L)$ where $W_\L \in \monoid$ and $f_\L \in\freeF$ is called a \emph{reduction system}. Let $S$ be a reduction system. Given $\L \in S$ and $L,R \in\monoid$, by the \emph{reduction} $\r_{L\L R}$ we mean the linear mapping $\freeF\longrightarrow\freeF$ that fixes all elements of $\monoid$ other than $LW_\L R$, and instead sends this basis element of $\freeF$ to the element $Lf_\L R$. A reduction  $\r_{L\L R}$ acts \emph{trivially} on an element $K$ of $\freeF$ if the coefficient of the basis element $LW_\L R$ in $K$ is zero.  If every reduction acts trivially on an element $K$, then $K$ is \emph{irreducible} (under $S$). We say that $K \in \freeF$ is \emph{reduction-finite} if, for every infinite sequence $\r_1, \r_2, \ldots$ of reductions, there exists $N \in \N$ such that $\r_i$ acts trivially on $(\r_{i-1} \circ \r_{i-2} \circ \cdots \circ \r_1)(K)$ for all $ i\geq N$. If $K$ is reduction-finite, then a \emph{final sequence} is any maximal finite sequence of reductions $\r_i$, such that each $\r_i$ acts nontrivially on $(\r_{i-1} \circ \r_{i-2}\circ \cdots \circ \r_1)(K)$. Additionally, if $K$ is reduction-finite and if its images under final sequences of reductions are the same, then we say that $K$ is \emph{reduction-unique}.   

A 5-tuple $(\L, \TT, W_1, W_2, W_3)$ where $\L, \TT \in S$ and $W_1, W_2, W_3 \in \monoid$\textbackslash \{\textit{I}\} is an \emph{overlap ambiguity} if $W_\L =W_1W_2$ and $W_\TT =W_2W_3$. This ambiguity is said to be \emph{resolvable} if there exist compositions of reductions $\r$ and $\r'$ such that $\r(f_\L W_3) = \r'(W_1f_\TT)$. Also, a 5-tuple $(\L, \TT, W_1, W_2, W_3)$ where $\L \ne \TT \in S$ and $W_1, W_2, W_3 \in\monoid$ is an \emph{inclusion ambiguity} if $W_\L = W_2$ and $W_\TT =W_1W_2W_3$. This ambiguity is said to be \emph{resolvable} if there exist compositions of reductions $\r$ and $\r'$ such that $\r(f_\TT) = \r'(W_1f_\L W_3)$. By an \emph{ambiguity} of $S$, we mean either an overlap ambiguity or an inclusion ambiguity.

\begin{theorem}[Diamond Lemma] \label{berglemma}
 Let $S$ be a reduction system on $\freeF$. Let $\scA$ be an algebra with generators in $\genset$ and relations $W_\L=f_\L$ for all $\L\in S$.  The following conditions are equivalent.
\begin{enumerate}
\item All ambiguities of S are resolvable.
\item All elements of $\freeF$ are reduction-unique under S.
\item The set of all irreducible words on $\genset$ with respect to $S$ form a basis for $\scA$. 
\end{enumerate}

\end{theorem}

\section{\normalsize The algebra $\U$}

We now consider the case when $\genset$ has only two elements $A$ and $B$. Given $q, r, s\in \F$, let $\scI_1=\scI_1(q,r,s)$ be the ideal of $\freeF$ generated by $AB-qBA-rA-sB$, and let $\U:=\freeF/\scI_1$.

\begin{proposition}\label{propoA^nB&AB^n}
For any $n \in \Z^+$, the identities 
\begin{eqnarray}
A^nB&=&B\sum_{t=0}^{n}\binom{n}{t}s^{t}(qA)^{n-t}+r\sum_{i=0}^{n-1}(qA+s)^{n-1-i}A^{i+1}, \label{EqA^nB} \\
AB^n&=&\sum_{t=0}^{n}\binom{n}{t}r^{t}(qB)^{n-t}A+s\sum_{i=0}^{n-1}(qB+r)^{n-1-i}B^{i+1} \label{EqAB^n} 
\end{eqnarray}
hold in $\U$.
\end{proposition}
\begin{proof} Both identities simply reduce to the defining relation $AB-qBA=rA+sB$ of $\U$ when $n=1$. If the given identities hold for some $n\in\Z^+$, then, with the goal of performing induction on $n$, at $n+1$, the desired left-hand sides may be obtained by multiplying $A$ from the left or by $B$ from the right. In the resulting right-hand sides, the identity $AB=qBA+rA+sB$ may be used such that, after a finite number of steps, the desired linear combinations of words will appear in the new right-hand sides. By induction on $n$, the desired identities are indeed true in $\U$.
\end{proof}

By introducing a new letter $C:=[A,B]=AB-BA$, the algebra $\U$ would consequently have the following presentation.
\begin{lemma}\label{lemmapresentation1}
The algebra $\U$ has a presentation by generators $A, B, C$ and relations
\begin{eqnarray}
    AB-qBA&=&rA+sB, \label{mainsubrel1}\\
    C&=&AB-BA. \label{mainsubrel2}
\end{eqnarray}
\end{lemma}
\begin{proof} Given $\scY=\{A,B, C\}$, let $\scI_2$ be the ideal of $\freeY$ generated by $ AB-qBA-rA-sB$ and $ C-AB+BA$. Since the generators in the respective presentations for $\freeF/\scI_1$ and $ \freeY/\scI_2$ satisfy the relations of the other, a routine argument may be used to show that there exists an algebra isomorphism $\freeF/\scI_1\longrightarrow \freeY/\scI_2$ which maps $A \mapsto A$, $B \mapsto B$ and $[A,B] \mapsto C$.
\end{proof}

We will often refer to some $q$- special relations from \cite[Appendix C]{hel00} such as the following. For a given $n \in \N$ and $z \in \F$,
\begin{eqnarray}
\{n\}_z&:=& \sum_{t=0}^{n-1}z^t, \label{qspecialeq1}\\
 (1-z)\{n\}_z &=& 1-z^{n}, \label{qspecialeq5}
\end{eqnarray}

If $n \leq 0$, then we interpret \eqref{qspecialeq1} as the empty sum $0$. 

\begin{lemma}\label{lemmaforAC^nequivwithx2} Let $\x_2=\x_2(q,s):=AC-qCA-sC$. For any $h\in \Z^+$, 
\begin{eqnarray}
    \sum_{i=1}^{h}q^{i-1}C^{i-1}\x_2C^{h-i}&=&AC^{h}-q^hC^{h}A-\{h\}_qsC^{h}. \label{summationforAC^nequivwithx2}
\end{eqnarray}
\end{lemma}
\begin{proof} The case $h=1$ is simply the definition of $\x_2$. Suppose \eqref{summationforAC^nequivwithx2} is true for some $h\in \N$. Multiplying both sides by $C$ from the right, the resulting right-hand side is a linear combination of the words $AC^{h+1}$, $C^hAC$ and $C^{h+1}$, where the $AC$ in $C^hAC$ may be replaced using the relation $AC=\x_2+qCA+sC$, which is immediate from the definition of $\x_2$. Adding $q^hC^{h}\x_2$ to both sides, we find that the identity is true at $h+1$. The desired result follows by induction on $h$.
\end{proof}

Succeeding computations will involve division by a power of $q$ or by a field element of the form $1-q^m$ for some nonzero integer $m$. Thus, we need the following.

\begin{assumption}\label{qFassume} The field $\F$ has characteristic zero, and the scalar $q$ is nonzero, and is not a root of unity.
\end{assumption}

\begin{proposition}\label{proppresentation3gen} The algebra $\U$ has a presentation by generators $A, B, C$ and relations 
\begin{eqnarray}
    AB&=&\frac{rA+sB-qC}{1-q}, \label{proprel1}\\
    AC&=&qCA+sC, \label{proprel2} \\
    BA&=&\frac{rA+sB-C}{1-q}, \label{proprel3}\\
    CB&=&qBC+rC, \label{proprel4}\\
    BC^kA&=&\frac{q^k rC^{k}A+q^ksBC^{k} +\{k\}_q rsC^{k}-C^{k+1}}{q^{k}(1-q)},\quad k\in \Z^+. \label{proprel5} 
    \end{eqnarray}
\end{proposition}
\begin{proof}
We view the left-hand side and right-hand side expressions of equations \eqref{proprel1} to \eqref{proprel5} as elements of $\freeFtwo$ and define
\begin{eqnarray}
\x_1&:=& AB-\frac{rA+sB-qC}{1-q}, \label{generatorx1} \\
\x_2&:=&AC-qCA-sC, \label{generatorx2}  \\
\x_3&:=&BA-\frac{rA+sB-C}{1-q}, \label{generatorx3}  \\
\x_4&:=&CB-qBC-rC, \label{generatorx4} \\
\x_5(k)&:=&BC^kA-\frac{q^k rC^{k}A+q^ksBC^{k} +\{k\}_q rsC^{k}-C^{k+1}}{q^{k}(1-q)},\quad k\in \Z^+. \label{generatorx5}
\end{eqnarray}  
Also, we denote generators of $\scI_2$ by
\begin{eqnarray}
    \z_1&=&AB-qBA-rA-sB, \label{generatorz1} \\
    \z_2&=&C-AB+BA. \label{generatorz2}
\end{eqnarray}
Let $\scI_3$ be the ideal of $\freeFtwo$ generated by 
\begin{equation} \label{elementsofV}
    \{\x_1, \x_2, \x_3, \x_4, \z_1, \z_2\}\cup\{\x_5(k)\}:  k \in \{1, 2, 3, \ldots\}.
\end{equation}
We claim that $\scI_2=\scI_3$. The relations \eqref{generatorx1}  and \eqref{generatorx3} may be used in some routine computations to obtain
\begin{eqnarray}
\x_3-\x_1&=&C-AB+BA, \label{equationx3-x1}\\
    \x_1 -q\x_3&=& AB-qBA-rA-sB, \label{equationx1-qx3} 
\end{eqnarray}
provided $q\neq 1$. With the use of \eqref{equationx3-x1} and \eqref{equationx1-qx3}, we have $\z_1, \z_2 \in \scI_3$. Thus we have $\scI_2 \subseteq \scI_3$. Next we show that $\scI_3 \subseteq \scI_2$. Observe that 
\begin{eqnarray}
    \z_1+q\z_2&=&AB-qBA-rA-sB+qC-qAB+qBA, \nonumber \\
   &=&(1-q)AB-rA-sB+qC, \nonumber \\
   \frac{\z_1+q\z_2}{1-q}&=&AB-\frac{rA+sB-qC}{1-q}. \label{equationforx1}
\end{eqnarray}
From \eqref{equationforx1}, we can easily derive
\begin{eqnarray}
    AB&=&\frac{\z_1+q\z_2}{1-q}+\frac{rA+sB-qC}{1-q}.
\end{eqnarray}
By further routine computations,
\begin{eqnarray}
  A\z_1-\z_1 A+A\z_2-q\z_2 A &=&AC-qCA-sC, \label{equationforx2}\\
\frac{\z_2+\z_1}{1-q}&=&BA-\frac{rA+sB-C}{1-q},\label{equationforx3}\\
   \z_1B-B\z_1+\z_2B-qB\z_2&=&CB-qBC-rC. \label{equationforx4} 
\end{eqnarray}
Equations \eqref{equationforx1} and \eqref{equationforx3} clearly show that $\x_1$ and $\x_3$ are linear combinations of $\z_1$ and $\z_2$, and so we have $\x_1, \x_3 \in \scI_2$. Also, because of absorbing property of ideals, equations \eqref{equationforx2} and \eqref{equationforx4} suggest that  $\x_2, \x_4 \in \scI_2$.

From  Lemma \ref{lemmaforAC^nequivwithx2}, we can easily obtain
\begin{eqnarray}
AC^{k}&=&\sum_{i=1}^{k}q^{i-1}C^{i-1}\x_2C^{k-i}+q^kC^{k}A+\{k\}_qsC^{k}.     
\end{eqnarray}
Since we have established that $\x_2, \x_3 \in \scI_2$, then $\x_3 C^n, B\sum_{i=1}^{k}q^{i-1}C^{i-1}\x_2C^{k-i} \in \scI_2$. Observe that
\begin{eqnarray}
    (1-q)\x_3 C^k&=&(1-q)BAC^k-rAC^k-sBC^{k}+C^{k+1}, \nonumber \\
   &=&(1-q)BAC^k-r\left(\sum_{i=1}^{k}q^{i-1}C^{i-1}\x_2C^{k-i}+q^kC^{k}A+\{k\}_qsC^{k}\right) \nonumber \\
   &&-sBC^{k}+C^{k+1}, \nonumber \\
    &=&(1-q)BAC^k-r\sum_{i=1}^{k}q^{i-1}C^{i-1}\x_2C^{k-i}-q^k rC^{k}A-\{k\}_q rsC^{k} \nonumber \\
    &&-sBC^{k}+C^{k+1}, \nonumber
\end{eqnarray}
and transposing the summation to the left-hand side, we obtain
\begin{eqnarray}
    (1-q)\x_3 C^k+r\sum_{i=1}^{k}q^{i-1}C^{i-1}\x_2C^{k-i}&=&(1-q)BAC^k-q^k rC^{k}A-\{k\}_q rsC^{k}-sBC^{k}+C^{k+1}. \nonumber
\end{eqnarray}
Also with Lemma \ref{lemmaforAC^nequivwithx2}, we have 
\begin{eqnarray}
(1-q)B\sum_{i=1}^{k}q^{i-1}C^{i-1}\x_2C^{k-i}&=&(1-q)BAC^{k}-(1-q)q^kBC^{k}A-(1-q)\{k\}_qsBC^{k}. \nonumber
\end{eqnarray}
The previous two identities, together with the earlier one with left-hand side $AC^{k}$, may be used in routine computations to show that  
\begin{equation}
 (1-q)\x_3 C^k+r\sum_{i=1}^{k}q^{i-1}C^{i-1}\x_2C^{k-i}- (1-q)B\sum_{i=1}^{k}q^{i-1}C^{i-1}\x_2C^{k-i},
\end{equation}
is equal to the linear combination
\begin{eqnarray}
(1-q)q^kBC^{k}A-q^k rC^{k}A-q^ksBC^{k} -\{k\}_q rsC^{k}+C^{k+1},\nonumber 
\end{eqnarray}
and this implies that $\x_5(k)$ for any $k$ is a linear combination of elements in $\scI_2$ and that $\x_5(k)\in \scI_2$. Since the remaining generators $\z_1$ and $\z_2$ of $\scI_3$ are precisely the generators of $\scI_2$, we now have $\scI_3 \subseteq \scI_2$. Thus, we have $\U=\freeFtwo/\scI_3$.  
\end{proof}


\section{\normalsize A basis for $\U$}

A basis of an algebra holds essential information for understanding its algebraic structure \cite[p. 10]{ufn95}. 
So we choose a basis for $\U$ based from its presentation on Proposition \ref{proppresentation3gen} using Bergman's Diamond Lemma.

We use relations of $\U$ given in Proposition \ref{proppresentation3gen} to construct a reduction system in $\freeY$. Let
\begin{eqnarray}
\S_1&=&\left( AB, \frac{rA+sB-qC}{1-q}\right), \label{reduction1} \\
\S_2&=&\left(   AC , qCA+sC\right), \label{reduction12}\\
\S_3&=& \left( BA, \frac{rA+sB-C}{1-q} \right), \label{reduction3}\\
\S_4&=&\left( CB, qBC+rC\right), \label{reduction4} \\
\TT_k&=& \left( BC^kA, \frac{q^k r C^{k}A+q^ksBC^{k} +\{k\}_q rsC^{k}-C^{k+1}}{q^{k}(1-q)} \right),\quad k\in \Z^+. \label{reduction5k}
\end{eqnarray}
Then $R:=\{\S_i, \TT_k\  :\   i\in \{1, 2,3, 4\} , k \in \{1, 2, 3 \ldots\}\}$
is a reduction system in $\freeY$ for $\U$ in three generators. In order to use an implication in Bergman's Diamond Lemma, first we show that any ambiguity of $R$ is resolvable. It is routine to show that there is no inclusion ambiguity given the reduction system $R$. In addition, all overlap ambiguities that do not involve an element $\TT_k \in R$ are
\begin{eqnarray}
\P_1&=&(\S_1, \S_3, A, B, A), \\
\P_2&=&(\S_2, \S_4, A, C, B), \\
\P_3&=&(\S_3, \S_1, B, A, B), \\
\P_4&=&(\S_3, \S_2, B, A, C), \\
\P_5&=&(\S_4, \S_3, C, B, A),
\end{eqnarray}
while all the overlap ambiguities that depend on an integer parameter ($k$) are
\begin{eqnarray}
\P_{6}(k)&=&(\S_1, \TT_k, A, B, C^kA), \\
\P_{7}(k)&=&(\S_4,  \TT_k, C, B, C^kA), \\
\P_{8}(k)&=&(\TT_k, \S_1, BC^k, A,B), \\
\P_{9}(k)&=&(\TT_k, \S_2, BC^k, A, C).
\end{eqnarray}

\begin{lemma}\label{lemmaAC^mandC^mB}
    For any $m\in \Z^+$, the following hold in $\U$.
    \begin{eqnarray}
        AC^m&=&q^mC^mA+\{m\}_q sC^{m}, \label{equationlemmaAC^m} \\ 
        C^m B&=& q^mBC^m+\{m\}_q rC^{m}. \label{equationlemmaC^mB} 
    \end{eqnarray}
\end{lemma}
\begin{proof} The case $m=1$ leads to precisely the equations \eqref{proprel2}  and \eqref{proprel4} in Proposition~\ref{proppresentation3gen}. By an argument similar to that done in Proposition~\ref{propoA^nB&AB^n}, the desired identities follow by induction on $m$.
\end{proof}

\begin{lemma} \label{lemmaallambiguityresolvable}
    All ambiguities of $R$ are resolvable.
\end{lemma}
\begin{proof}
We prove this lemma directly by determining compositions of reductions $\r_i$ and $\r'_i$ for each ambiguity $\P_1, \ldots, \P_5$ and $\P_6(k), \ldots, \P_9(k)$ that will satisfy condition for resolvable ambiguity. For any positive integer $k$ and any $U, W \in \left< \scY\right>$, we let
\begin{eqnarray}
\a_{(k,W)}&:=&\r_{C^{k-1}\S_2W}\circ \r_{C^{k-2}\S_2 CW}\circ \r_{C^{k-3}\S_2 C^2W}\circ \cdots \circ \r_{C\S_2C^{k-2}W}\circ \r_{\S_2C^{k-1}W}, \nonumber \\
\b_{(k,U)}&:=&\r_{UC^{k-1}\S_4}\circ \r_{UC^{k-2}\S_4 C}\circ \r_{UC^{k-3}\S_4 C^2}\circ \cdots \circ \r_{UC\S_4C^{k-2}}\circ \r_{U\S_4C^{k-1}}.\nonumber 
\end{eqnarray}
With Lemma \ref{lemmaAC^mandC^mB}, we take note of the following simple results:
\begin{eqnarray}
\a_{(k,W)}(AC^kW)&=&q^kC^kAW+\{k\}_q s C^{k}W, \\
\b_{(k,U)}(UC^kB)&=&q^k UBC^k+\{k\}_q r UC^{k}.
\end{eqnarray}
For simpler notation, given $\m \in R$,  we write $\r_{I\m I}, \r_{U\m I}, \r_{I\m V}$ as $\r_{\m}, \r_{U\m}, \r_{\m V}$ respectively. 

We first consider $\P_1=(\S_1, \S_3, A, B, A)$. Notice that $ABA$ is precisely the nontrivial word involved with this overlap ambiguity. Let $\r_1=\r_{\S_3}$ and $\r'_1=\r_{\S_2}\circ \r_{\S_1}$.  Observe that

\begin{eqnarray}
    \r_1(f_{\S_1}A)  &=&\r_{\S_3}\left(\frac{rA^2+sBA-qCA}{1-q}\right), \nonumber \\
    &=&\frac{rA^2+s\left(\frac{rA+sB-C}{1-q}\right)-qCA}{1-q}, \nonumber \\
    &=&\frac{(1-q)rA^2+rsA+s^2B-sC-(1-q)qCA}{(1-q)^2}, \nonumber 
\end{eqnarray}
and 
\begin{eqnarray}
   \r'_1(Af_{\S_3})   &=&(\r_{\S_2}\circ \r_{\S_1})\left(\frac{rA^2+sAB-AC}{1-q}  \right), \nonumber \\
&=&\r_{\S_2}\left(\frac{rA^2+s\left(\frac{rA+sB-qC}{1-q}\right)-AC}{1-q}  \right), \nonumber \\
   &=&\frac{rA^2+s\left(\frac{rA+sB-qC}{1-q}\right)-(qCA+sC)}{1-q}, \nonumber \\
   &=&\frac{(1-q)rA^2+rsA+s^2B-sC-(1-q)qCA}{(1-q)^2}, \nonumber \\
&=&\r_1(f_{\S_1}A).\nonumber
\end{eqnarray}

Thus, for the ambiguity $\P_1=(\S_1, \S_3, A, B, A)$, if $\r_1=\r_{\S_3}$ and $\r'_1=\r_{\S_2}\circ \r_{\S_1}$, then we have $\r'_1(Af_{\S_3})=\r_1(f_{\S_1}A)$, which implies resolvability of the ambiguity. To complete the proof, we check all other ambiguities. The process involves routine computations like those done above for $\P_1$. We only summarize below what compositions of reductions are used for each ambiguity, which lead to the desired resolvability condition, like the equation $\r'_1(Af_{\S_3})=\r_1(f_{\S_1}A)$ for $\P_1$. Again, such equations may be verified routinely for each of the remaining ambiguities.

\begin{enumerate}\item For $\P_2=(\S_2, \S_4, A, C, B)$, if $\r_2=\r_{\S_4} \circ \r_{C\S_1}$ and $\r'_2=\r_{\S_2}\circ \r_{\S_1 C}$, then $ \r_2(f_{\S_2}B)= \r'_2(Af_{\S_4})$.
\item For $\P_3=(\S_3, \S_1, B, A, B)$, if $\r_3=\r_{\S_4} \circ \r_{\S_1}$ and $\r'_3=\r_{\S_3}$, then $ \r_3(f_{\S_3}B)= \r'_3(Bf_{\S_1})$.
\item For $\P_4=(\S_3, \S_2, B, A, C)$, if $\r_4=\r_{\S_2} $ and $\r'_4=\r_{\TT_1}$, then $ \r_4(f_{\S_3}C)= \r'_4(B f_{\S_2})$.
\item For $\P_5=(\S_4, \S_3, C, B, A)$, if $\r_5=\r_{\TT_1} $ and $\r'_5=\r_{\S_4}$, then $\r_5(f_{\S_4}A)=\r'_5(Cf_{\S_3})$.
\item For $\P_6(k)=(\S_1, \TT_k, A, B, C^kA)$, if $\r_6=\r_{\TT_k}\circ \a_{(k, A)} $ and $\r'_6= \a_{(k+1, I)} \circ \a_{(k, I)}\circ \r_{\S_1 C^k}\circ \a_{(k, A)}$, then $\r_6(f_{\S_1}C^kA)=\r'_6(A f_{\TT_k})$.
\item For $\P_7(k)=(\S_4,  \TT_k, C, B, C^kA)$, if $\r_7=\r_{\TT_{k+1}} $ and $\r'_7= \r_{\S_4}$, then $\r_7(f_{\S_4}C^kA)= \r'_7(C f_{\TT_k})$.
\item For $\P_8(k)=(\TT_k, \S_1, BC^k, A,B)$, if $\r_8=\b_{(k+1, I)} \circ \b_{(k, I)} \circ \b_{(k, B)}\circ  \r_{C^k\S_1} $ and $\r'_8=\b_{(k, B)} \circ \r_{\TT_k}$, then $\r_8(f_{\TT_k}B)=\r'_8(BC^kf_{\S_1})$.
\item For  $\P_{9}(k)=(\TT_k, \S_2, BC^k, A, C)$, if $\r_9=\r_{C^k\S_2} $ and $\r'_9= \r_{\TT_{k+1}}$, then $\r_9(f_{\TT_k}C)=\r'_9(BC^kf_{\S_2})$.
\end{enumerate}
These results suggest that with $\r_i$ and $\r'_i$ for $i \in \{1, 2, \ldots, 9\}$, all ambiguities $\P_1, \ldots, \P_5$ and $\P_6(k)$, $\ldots$ , $\P_9(k)$ are resolvable. This completes the proof.
\end{proof}

\begin{theorem}\label{theobasisfor3gen}
The elements
\begin{eqnarray}
  B^l C^m, C^m A^t,\qquad (l, m  \in \N,\   t\in \Z^+),
\end{eqnarray}
form a basis for $\U$.
\end{theorem}
\begin{proof} 
We consider $R$ whose elements are given by \eqref{reduction1}-\eqref{reduction5k} based from the defining relations of $\U$ as previously stated in Proposition \ref{proppresentation3gen}. We first show that
\begin{equation} \label{basisforalgebrain3gen}
    \{B^h C^j,  C^j A^k\  :\    h, j \in \N, k\in \Z^+\},
\end{equation}
is the set of all irreducible words under $R$. Notice that collection \eqref{basisforalgebrain3gen} is clearly a set of irreducible words with respect to the reduction system $R$ since words $AB, AC, BA, CB$ and $BC^kA$ with $k\in \Z^+$ do not appear as a subword in any of its elements. Suppose $W$ is not in \eqref{basisforalgebrain3gen}. Then $W$ must have a subword of the form $A^xC^yB^z$ or $B^uC^yA^w$ where $ x, y, z  \in \N, u, w \in \Z^+$ and at most one of the nonnegative powers  is equal to zero. It is clear that we cannot have two or three variables among $x, y$ and $z$ to be zero for $A^xC^yB^z$ because it will contradict our supposition. This means that we only have to consider cases when $x=0$, $y=0$, $z=0$, and when $x, y, z \in \Z^+$ for $A^xC^yB^z$. Meanwhile, we have cases $y=0$, and $y\neq 0$ for  $B^uC^yA^w$. 
If $x=0$, then a reduction which involves $\S_4$ would act nontrivially on  $A^xC^yB^z=C^yB^z$. If $y=0$, a reduction which involves $\S_1$ would act nontrivially on  $A^xC^yB^z=A^xB^z$,  while a reduction which involves $\S_3$ would act nontrivially on  $B^uC^yA^w=B^uA^w$. And if $z=0$, a reduction which involves $\S_2$ would act nontrivially on $A^xC^yB^z=A^xC^y$. For the case $x, y, z \in\Z^+$, reductions which involve $\S_2$ or  $\S_4$ would act nontrivially on $A^xC^yB^z$ ,  while reductions which involve $\TT_y$ would act nontrivially on  $B^uC^y\A^w$ when $y \neq 0$.

It is clear that in any of the mentioned cases, $W$ is not irreducible. Thus, any irreducible element with respect to the the reduction system $R$ are  in \eqref{basisforalgebrain3gen}.
Now, we only need to show that elements in \eqref{basisforalgebrain3gen} form a basis for $\U$. To do this, we invoke Bergman's Diamond Lemma. The only implication we
need from the Diamond Lemma is that: if all the ambiguities of a reduction
system $S$ are resolvable and if $\scK$ is the ideal of $\freeF$ generated by all  $W_{\S}-f_{\S} (\S \in S)$, then the images of all the $S$-irreducible words under the canonical map $\freeF \longrightarrow \freeF/\scK$ form a basis for $\freeF/\scK$. If we take $S=R$, $\scK=\scI_3$ generated by expressions in \eqref{generatorx1}-\eqref{generatorz2}, then with Lemma \ref{lemmaallambiguityresolvable}, the elements in \eqref{basisforalgebrain3gen} form a basis for $\U$.
\end{proof}

\section{\normalsize The algebra $\Ur$}
In this section, we shall consider the case $r\neq 0$ and $s=0$ for the algebra $\U$, that is, the algebra $\Ur$ that has presentation by generators $A, B$ and relation
\begin{eqnarray}
  AB-qBA=rA. \label{rel1Ur}  
\end{eqnarray}

To proceed, we again make use of the additional generator $C=[A,B]=AB-BA$ so that algebra $\Ur$ would have a presentation by generators $A, B, C$ and relations
\begin{eqnarray}
    AB&=&\frac{rA-qC}{1-q}, \label{mainrel1Ur}\\
    AC&=&qCA, \label{mainrel2Ur} \\
    BA&=&\frac{rA-C}{1-q}, \label{mainrel3Ur}\\
    CB&=&qBC+rC, \label{mainrel4Ur}\\
    BC^kA&=&\frac{q^k rC^{k}A -C^{k+1}}{q^{k}(1-q)},\quad k\in \Z^+, \label{mainrel5Ur} 
    \end{eqnarray}
which follows directly from Proposition \ref{proppresentation3gen}.

Let $n\in\Z^+$. Routine induction using \eqref{rel1Ur} results to
   \begin{eqnarray}
        A^nB&=&q^nBA^n+\{n\}_q rA^n, \label{EqA^nBlemmaUr}
    \end{eqnarray}
both sides of which, we multiply by $B^{n-1}$ from the right. The resulting right-hand side is a linear combination of only two words. The first term is $(q^nBA^n)B^{n-1}$. To the expressions in parentheses, we substiture using $q^nBA^n=A^nB-\{n\}_q rA^n$, which is just one equivalent form of \eqref{EqA^nBlemmaUr}. We now have
    \begin{eqnarray}
        A^nB^n&=&(q^{n}BA+\{n\}_qrA)A^{n-1}B^{n-1}, \label{EqA^nB^nUr}
    \end{eqnarray}
which we shall use to prove 
     \begin{eqnarray}
    A^nB^n&=&\prod_{i=0}^{n-1}(q^{n-i}BA+\{n-i\}_qrA). \label{EqA^nB^nproductUr}
 \end{eqnarray}
The case $n=1$ is simply the relation \eqref{rel1Ur}. Suppose \eqref{EqA^nB^nproductUr} is true for some $n\in\Z^+$. We multiply both sides of \eqref{EqA^nB^nUr} by $A$ from the left and $B$ from the right. The inductive hypothesis, with the aid of \eqref{rel1Ur}, may then be used on the resulting right-hand side, and \eqref{EqA^nB^nproductUr} follows.

Using the relation \eqref{mainrel1Ur} on \eqref{EqA^nB^nproductUr}, routine computations may be used to prove
    \begin{eqnarray}
        A^nB^n&=&\prod_{i=0}^{n}\frac{rA-q^{n+1-i}rC}{(1-q)}. \label{EqA^nB^ninACUr}
\end{eqnarray}
The relations \eqref{mainrel2Ur} and \eqref{mainrel4Ur} may be generalized into
    \begin{eqnarray}
        A^nC&=&q^{n}CA^n, \label{EqA^nCUr} \\
        AC^n&=& q^{n}C^nA, \label{EqAC^nUr}  \\
        C^nB&=&q^nBC^n+\{n\}_qrC^n,  \label{EqC^nBUr}   \\
       CB^n&=&\sum_{i=0}^{n}\binom{n}{i}r^{n-i}q^{i}B^{i}C \label{EqCB^nUr}, 
    \end{eqnarray}
using induction. For some important proofs that shall come later, we will need a generalization of \eqref{EqA^nCUr}--\eqref{EqCB^nUr} which is in the following.

\begin{proposition} \label{propoA^nC^mandC^mB^nUr}
  For any $n, m \in \Z^+$, 
  \begin{eqnarray}
A^nC^m&=&q^{nm}C^mA^n,  \label{EqA^nC^ms1summationUr} \\
C^mB^n&=&\sum_{i=0}^{n}\binom{n}{i}q^{mn-mi}(\{m\}_qr)^iB^{n-i}C^m.\label{EqC^mB^ns1summationUr}
  \end{eqnarray}
\end{proposition}
\begin{proof} The first identity may be obtained, by routine induction, from \eqref{EqA^nCUr}--\eqref{EqAC^nUr}. We proceed with the second identity. Let $m\in \Z^+$. The case $n=1$ is simple equivalent to \eqref{EqC^nBUr}. If \eqref{EqC^mB^ns1summationUr} is true for some $n\in \Z^+$, then $C^mB^{n+1}=\left(C^mB^n\right)B=\left(\sum_{i=0}^{n}\binom{n}{i}q^{mn-mi}(\{m\}_qr)^iB^{n-i}\right)(C^mB)$, where we use \eqref{EqC^nBUr} on the last parenthetical expression. After some routine computations, \eqref{EqC^mB^ns1summationUr} may be shown to be true for $n+1$. By induction, we get the desired result.
\end{proof}

Having introduced the above relations in $\Ur$ that shall be important in our proofs and computations, we now show an important consequence of Theorem~\ref{theobasisfor3gen}.

\begin{corollary}\label{rCor} If $r\neq 1$, then there is no algebra homomorphism $\Ur\into\UrOne$ that sends $A\mapsto A$ and $B\mapsto rB$. 
\end{corollary}
\begin{proof} We prove this by contraposition. Suppose that such a homomorphism $\nonHomo$ exists. An immediate routine consequence is that $\nonHomo(C)=rC$. Using the relations \eqref{mainrel1Ur} and \eqref{EqA^nC^ms1summationUr} in $\Ur$, we have\linebreak $CA^2\cdot BC=CA(AB)C=CA\frac{rA-qC}{1-q}C=\frac{r}{1-q}C(A^2C)-\frac{q}{1-q}C(AC^2)=\frac{rq^2}{1-q}C^2A^2-\frac{rq^3}{1-q}C^3A$, and so,
\begin{eqnarray}
\nonHomo\lpar CA^2\cdot BC\rpar = \frac{r^3q^2}{1-q}C^2A^2-\frac{r^4q^3}{1-q}C^3A.\label{nonHomo1}
\end{eqnarray}
In the codomain $\UrOne$, we have the computations $\nonHomo\lpar CA^2\rpar\cdot \nonHomo\lpar BC\rpar=rCA^2\cdot r^2BC=r^3CA(AB)C$, where the reordering formula for $AB$ in $\UrOne$ is $AB=\frac{A-qC}{1-q}$. By this identity, and also \eqref{EqA^nC^ms1summationUr},
\begin{eqnarray}
\nonHomo\lpar CA^2\rpar\cdot \nonHomo\lpar BC\rpar = \frac{r^3q^2}{1-q}C^2A^2-\frac{r^3q^3}{1-q}C^3A.\label{nonHomo2}
\end{eqnarray}
Since $\nonHomo$ is a homomorphism, the right-hand sides of \eqref{nonHomo1} and \eqref{nonHomo2} are equal, and by the linear independence of $C^2A^2$ and $C^3A$ from Theorem~\ref{theobasisfor3gen}, we may equate the scalar coefficients of $C^3A$ in \eqref{nonHomo1} and \eqref{nonHomo2} to obtain $r^3(r-1)=0$. Since $r$ is assumed to be nonzero, $r=1$.
\end{proof}

Corollary~\ref{rCor} gives us justification to proceed with retaining the parameter $r$ in the generalized $q$-deformed relation $AB-qBA=rA$, because, unlike in the undeformed case when the relations $AB-BA=rA$ (with $r\notin\{0,1\}$) and $AB-BA=A$ will lead to isomorphic algebras because of a homomorphism that performs $A\mapsto A$ and $B\mapsto rB$, Corollary~\ref{rCor} implies that this is not possible for the $q$-deformed case (with $q\neq 1$). 

\begin{theorem}\label{theoB^yC^kA^w}  
Any element of $\Ur$ of the form 
 \begin{eqnarray}
       B^y C^kA^w \nonumber
    \end{eqnarray}
where $k, w, y \in \Z^+$ is an element of $\Span\{B^lC^m, C^mA^l\  :\   l\in \N,\  m \in \Z^+\}$.
\end{theorem}
\begin{proof} Let $\scS:=\Span\{B^lC^m, C^mA^l\  :\   l\in \N,\  m \in \Z^+\}$, and let $k\in\Z^+$. Elements of the form $BC^kA$ are in $\scS$ because of the relation \eqref{mainrel5Ur}. In particular, $BC^kA\in\scS_0:=\Span\{C^mA^l\  :\   l\in \N,\  m \in \Z^+\}$. If, for some $w\in\Z^+$, we have $BC^kA^w\in\scS_0$, then we have an equation that expresses $BC^kA^w$ as a linear combination of basis elements of the form $C^mA^l$. We multiply both sides of this equation by $A$ from the right, and so, $BC^kA^{w+1}$ is a linear combination of elements of the form $C^mA^{l+1}$, and this proves that $BC^kA^{w+1}\in\scS_0$. By induction, $BC^kA^w\in\scS_0$, for all $w\in\Z^+$. We perform another induction, with the statement $BC^kA\in \scS_0$ as the basis step. Suppose that for some $w\in\Z^+$, we have $B^wC^kA^w\in\scS_0$. Thus, $B^wC^kA^w$ is a linear combination of elements of the form $C^mA^l$, and consequently, $B^{w+1}C^kA^{w+1}$ is a linear combination of elements of the form $BC^mA^{l+1}$, which have been proven earlier to be elements of $\scS_0$. By induction, $B^wC^kA^w\in\scS_0$ for all $w\in\Z^+$. An analogous induction argument may be used to show that $B^wC^kA^{w+n}\in\scS_0$ for all $n\in\N$. The remaining case is when the exponent of $B$ is strictly greater than the exponent of $A$ in the word $B^y C^kA^w$. We may write such a word as $B^{w+t} C^kA^w$, where $t\in\Z^+$. If $t=1$, then using the fact that $B^wC^kA^w\in\scS_0$, the element $B^wC^kA^w$ is a linear combination of elements of the form $C^mA^l$. Consequently, $B^{w+1}C^kA^w$ is a linear combination of elements of the form $BC^mA^l$, which have been proven earlier to be in $\scS_0\sub\scS$. If, for some $t\in\Z^+$, we have $B^{w+t} C^kA^w\in\scS$, then this element is a linear combination of elements of the form $B^\lambda C^\mu$ and $C^mA^l$. Thus, $B^{w+t+1} C^kA^w$ is a linear combination of elements of the form $B^{\lambda+1} C^\mu$ and $BC^mA^l$. The latter have been proven earlier to be in $\scS_0\sub\scS$, while the former are spanning set elements of $\scS$. Therefore, $B^{w+t+1} C^kA^w\in\scS$. This completes the proof.
\end{proof}

\begin{theorem} \label{theoC^kA^WB^yUr}
Given $k, w, y \in \Z^+$ and $t\in\N$,
\begin{eqnarray}
C^kA^wB^yC^t = \sum_{i=0}^{y}\binom{y}{i}q^{wy-wi}(\{w\}_qr)^i\sum_{t=0}^{y-i}\binom{y-i}{t}q^{k(y-i-t)+tw}(\{k\}_qr)^tB^{y-i-t}C^{k+t} A^w. \label{EqC^kA^wB^yasBCAterms0}
\end{eqnarray}
\end{theorem}
\begin{proof} From \eqref{EqA^nBlemmaUr}, we find that $A^wB=q^{w}BA^w +\{w\}_qr A^w$. Multiplying both sides by $C^k$ we find that the identity   
\begin{eqnarray}
C^kA^wB^y&=&\sum_{i=0}^{y}\binom{y}{i}q^{wy-wi}(\{w\}_qr)^iC^kB^{y-i}A^w, \label{EqC^kA^wB^y}
  \end{eqnarray} 
is true in $\Ur$ for the case $y=1$. The proof of \eqref{EqC^kA^wB^y} may be completed by induction, with some aid again from the equation $A^wB=q^{w}BA^w +\{w\}_qr A^w$.

Using equation \eqref{EqC^mB^ns1summationUr} in Proposition \ref{propoA^nC^mandC^mB^nUr}, we have
\begin{eqnarray}
 C^kB^{y-i}&=&\sum_{t=0}^{y-i}\binom{y-i}{t}q^{k(y-i)-kt}(\{k\}_qr)^tB^{y-i-t}C^k,\nonumber   
\end{eqnarray}
which we subtitute into \eqref{EqC^kA^wB^y} to obtain 
\begin{eqnarray}
C^kA^wB^y &=&\sum_{i=0}^{y}\binom{y}{i}q^{wy-wi}(\{w\}_qr)^i\sum_{t=0}^{y-i}\binom{y-i}{t}q^{k(y-i)-kt}(\{k\}_qr)^tB^{y-i-t}C^k A^w,\nonumber 
\end{eqnarray}
both sides of which, we multiply by $C^t$ from the right. The resulting right-hand side is a linear combination of words of the form $B^{y-i-t}C^k A^wC^t$ where the subword $A^wC^t$ may be replaced by $q^{wt}C^tA^w$ using \eqref{EqA^nC^ms1summationUr}. The result is \eqref{EqC^kA^wB^yasBCAterms0}, as desired. 
\end{proof}

\section{\normalsize Constructing Lie polynomials in $A, B$ of $\Ur$}

Let $\LiesubUr$ denote the Lie subalgebra of $\Ur$ generated by $A$ and $B$. Fix an element $U \in \LiesubUr$. We recall the linear map $\ad U$ that sends $V \mapsto [U,V]$ for any $V \in \LiesubUr$.
We exhibit some important elements of the Lie subalgebra $\LiesubUr$ generated by $A$ and $B$.



\begin{proposition}\label{nestAdProp}
  For any $m \in \Z^+$,
  \begin{eqnarray}
        (\ad C)^m A&=&(1-q)^mC^mA, \label{EqadC^mAUr}\\
    (\ad C)^m B&=&(q-1)^mBC^m+(q-1)^{m-1}rC^m. \label{EqadC^mBUr}
   \end{eqnarray} 
\end{proposition}
\begin{proof} Use induction on $m$, with routine computations that involve the relations \eqref{mainrel2Ur} and \eqref{mainrel4Ur}.
\end{proof}

\begin{lemma} \label{lemmaC^minLie}
For any $m\in\Z^+$, the following holds in $\Ur$:
\begin{eqnarray}
    q^{m}[C^mA, B]&=&\{m+1\}_qC^{m+1}. \label{Eq[B,C^mA]}
\end{eqnarray}
\end{lemma}
\begin{proof} Use \eqref{mainrel1Ur} and \eqref{mainrel5Ur} on $ [C^mA, B]=C^mAB-BC^mA$.
\end{proof}

\begin{proposition}\label{PropoelementsofLie1Ur}
For any $m\in \Z^+$,
\begin{eqnarray}
    C^m, BC^m, C^mA \in \LiesubUr.
\end{eqnarray} 
\end{proposition}
\begin{proof} Equation \eqref{EqadC^mAUr} and Lemma \ref{lemmaC^minLie} imply that all elements of the form $C^m$ for any $m\in \Z^+$ are in $\LiesubUr$. Routine use of \eqref{EqA^nCUr}--\eqref{EqC^nBUr} yields $[C^m, A]=(1-q^m)C^mA$, $[B, C^m]=(1-q^m)BC^m-\{m\}_q rC^{m}$. Isolating the terms with $BC^m$ and $C^mA$, respectively, we find that $BC^m, C^mA \in \LiesubUr$ since $C^m$, $[C^m, A]$, $[B, C^m]  \in \LiesubUr$.
\end{proof}

\begin{lemma} \label{lemmaadA(CmA&BCmUr}
Fix $m \in \Z^+$. For any $n\in \Z^+$, 
\begin{eqnarray}
    (-\ad A)^n(C^mA)&=&(1-q^m)^nC^mA^{n+1}, \label{EqadA^nCmAUr} \\
    (\ad B)^n(BC^m)&=&\left((1-q^m)B-\{m\}_qr\right)^nBC^m. \label{EqadA^nBCmUr}
\end{eqnarray}
\end{lemma}
\begin{proof} Use \eqref{EqAC^nUr} and \eqref{EqC^nBUr} and induction on $n$. 
\end{proof}

\begin{theorem}\label{TheThm}
The following elements 
\begin{eqnarray}
    A, B, B^nC^m, C^mA^k,\qquad (m, k \in \Z^+, n \in \N), \label{basisforLUr}
\end{eqnarray}
form a basis for Lie algebra $\LiesubUr$.
\end{theorem}
\begin{proof}
Let $\scK$ be the span of the elements in \eqref{basisforLUr}. Notice that elements in $\scK$ are linearly independent as they are basis elements of $\Ur$. To show that $\scK$ is equal to $\LiesubUr$, we only have to show that the following conditions are satisfied:
\begin{enumerate}
    \item\label{LieSubI} $A, B \in \scK$,
    \item\label{LieSubII} $\scK$ is a Lie subalgebra of $\Ur$,
    \item\label{LieSubIII} $\scK \subseteq \LiesubUr$.
\end{enumerate}
 The condition \ref{LieSubI} immediately follows from the definition of $\scK$. For condition \ref{LieSubII}, we show that $\scK$ is closed under the Lie bracket operation. That is, we show that for any basis elements $L, R$ of $\scK$ from\eqref{basisforLUr}, $[L, R]$ is a linear combination of the elements in \eqref{basisforLUr}. Given $t, u, v, w \in \Z^+$ and $x, y \in \N$, routine use of \eqref{EqA^nCUr}--\eqref{EqCB^nUr} and Proposition~\ref{propoA^nC^mandC^mB^nUr} gives us
\begin{eqnarray}
\left[C^u, C^t \right]&=&0,\label{4''} \\
  \left[C^u, A\right] &=&(1-q^u)C^uA, \label{2'}  \\
\left[ C^uA^w, A\right]&=&(1-q^u)C^uA^{w+1},\label{5} \\
  \left[C^uA^w, C^t \right]&=& (q^{tw}- 1)C^{t+u}A^w, \label{7'}  \\
\left[C^uA^w, C^tA^v \right]&=&(q^{tw}-q^{uv})C^{t+u}A^{v+w}, \label{8}  \\
\left[C^u, B \right]&=&(q^u-1)BC^u+\{u\}_q rC^{u}, \label{3'}\\
\left[B^tC^u, B \right]&=&(q^u-1)B^{t+1}C^u+\{u\}_q rB^tC^{u}, \label{3}\\
  \left[C^u, B^wC^t \right] &=&\sum_{i=0}^{w}\binom{w}{i}q^{uw-ui}(\{u\}_qr)^iB^{w-i}C^{t+u}- B^wC^{t+u}, \label{4'} \\
   \left[B^vC^u, B^wC^t \right]  &=&\sum_{i=0}^{w}\binom{w}{i}q^{uw-ui}(\{u\}_qr)^iB^{v+w-i}C^{t+u} \nonumber \\
    &&-\sum_{i=0}^{v}\binom{v}{i}q^{tv-ti}(\{t\}_qr)^iB^{w+v-i}C^{t+u},   \label{4} \\
\left[ B^tC^u, A\right] &=&B^tC^uA-\sum_{i=0}^{t}\binom{t}{i}q^{t+u-i}r^{i}B^{t-i}C^uA, \label{2}\\
 \left[ C^uA^w, B\right]&=&C^uA^wB-BC^uA^w, \label{6}  \\
  \left[C^uA^w, B^yC^t \right]&=& C^uA^wB^yC^t- B^yC^{t+u}A^w.\label{7}
\end{eqnarray}
The commutation relations of the basis elements of $\scK$ in \eqref{basisforLUr} are summarized in the following table.

\begin{center}
\begin{tabular}{ |c||c|c|c|c| } 
 \hline
 $[ \cdot ,\cdot ]$ & $A$ & $B$ & $B^yC^t$ &$C^t A^v$ \\ 
 \hline\hline
  $A$ & $0$ &  & &   \\ 
 \hline 
 $B$ & $-C$ & $0$ & &    \\ 
 \hline 
 $B^xC^u$ & \eqref{2}, \eqref{2'} & \eqref{3}, \eqref{3'} & \eqref{4}, \eqref{4'}, \eqref{4''}    &   \\ 
 \hline 
 $C^u A^w$  & \eqref{5} & \eqref{6} &\eqref{7}, \eqref{7'} & \eqref{8}  \\ 
 \hline 
\end{tabular}
\end{center}

For each of the right-hand sides of the relations \eqref{4''}--\eqref{4}, we find that the result of the Lie bracket is a linear  combination of \eqref{basisforLUr}, and is hence in $\scK$. For the last three relations \eqref{2}--\eqref{7}, we use Theorems~\ref{theoB^yC^kA^w} and \ref{theoC^kA^WB^yUr} to deduce that the right-hand sides are also linear  combinations of \eqref{basisforLUr}. Hence, condition \ref{LieSubII} is satisfied.  

To prove \ref{LieSubIII}, we show that every basis element of $\scK$ is in $\LiesubUr$. The basis elements $A, B, C$ have this property by the definition of $\LiesubUr$. The rest of the basis elements of $\scK$ are in $\LiesubUr$ because of Lemma~\ref{lemmaadA(CmA&BCmUr}. At this point, we have proven $\scK\subseteq \Liesub$. This completes the proof.
\end{proof}

\section{\normalsize Lie polynomial characterization in $\Us$}

We now consider the case $r=0$ and $s\neq 0$ for the algebra $\U$. That is, the algebra $\Us$ with a presentation by generators $A, B$ and relation
\begin{eqnarray}
  AB-qBA=sB. \label{rel1Us}  
\end{eqnarray}
The Lie subalgebra of $\Us$ generated by $A$ and $B$ shall be denoted by $\LiesubUs$.  In this section, we will be comparing properties of the algebras $\Uss$ and $\Us$, and of the corresponding Lie subalgebras $\LiesubUss$ and $\LiesubUs$. First we look at the algebra structure. In particular, our goal is to show that there is a Lie algebra isomorphism $\LiesubUss\into\LiesubUs$ that sends $A\mapsto\beta B$ and $B\mapsto\alpha A$, for some nonzero $\alpha,\beta\in\F$. To accomplish this, we start with the analogs of the relations \eqref{mainrel1Ur}--\eqref{mainrel5Ur} for $\Ur$, which are
\begin{eqnarray}
    AB&=&\frac{sB-qC}{1-q}, \label{mainrel1Us}\\
    AC&=&qCA+sC, \label{mainrel2Us} \\
    BA&=&\frac{sB-C}{1-q}, \label{mainrel3Us}\\
    CB&=&qBC, \label{mainrel4Us}\\
    BC^kA&=&\frac{q^ksBC^{k} -C^{k+1}}{q^{k}(1-q)},\quad (k\in \Z^+), \label{mainrel5Us} 
    \end{eqnarray}
for $\Us$. By routine induction, the relations \eqref{mainrel2Us} and \eqref{mainrel4Us} may be generalized into
\begin{eqnarray}
A^nC^m&=&\sum_{i=0}^{n}\binom{n}{i}q^{mn-mi}(\{m\}_qs)^iC^mA^{n-i}, \label{EqA^nC^mUs} \\
C^mB^n&=&q^{mn}B^{n}C^m,\label{EqC^mB^nUs}
  \end{eqnarray}
respectively, for any $m,n\in\Z^+$. Routine induction using \eqref{rel1Us} yields
\begin{eqnarray}
AB^n&=&q^nB^{n}A+\{n\}_qsB^{n},\quad(n\in\Z^+). \label{EqAB^nUs}
\end{eqnarray}

An analog of Corollary~\ref{rCor}, which was for $\Ur$, will now be stated and proven for $\Us$. This serves as justification for retaining the parameter $s$ in $AB-qBA=sB$.

\begin{corollary}\label{sCor} If $s\neq 1$, then there is no algebra homomorphism $\Us\into\UsOne$ that sends $A\mapsto sA$ and $B\mapsto B$.
\end{corollary}
\begin{proof} If such a map $\nonHomo$ exists, then from $\nonHomo\lpar CA\cdot B^2C\rpar-\nonHomo\lpar CA\rpar\nonHomo\lpar B^2C\rpar=0$, routine computations, using the relations \eqref{mainrel1Us} and \eqref{EqC^mB^nUs}, may be made to show that the basis element $B^2C^2$ appears on the left-hand side. We note here that \eqref{mainrel1Us} applies to computations in the domain, but this becomes $AB=\frac{B-qC}{1-q}$ for computations in the codomain. By a linear independence argument, equating the coefficient of $B^2C^2$ to zero yields $s=1$.\blitza 
\end{proof}
\begin{corollary}\label{noIsoCor} For any nonzero $s\in\F$, there is no algebra homomorphism $\Uss\into\Us$ that sends $A\mapsto\beta B$ and $B\mapsto\alpha A$, where $\alpha,\beta\in\F$ are nonzero.
\end{corollary}
\begin{proof} Suppose such an algebra homomorphism exists, and apply this to both sides of the relation\linebreak $AB-qBA=sA$ for the algebra $\Uss$ to obtain $\alpha\beta BA-q\alpha\beta AB-s\beta B=0$, on which, we use the reordering formula $AB=qBA+sB$ in the algebra $\Us$. The result is $(q^2-1)BA+\frac{s}{\alpha}\cdot (\alpha q+1) B=0$. By Assumption~\ref{qFassume}, $q$ is not a root of unity, so we may divide both sides by $q^2-1$ to obtain $BA+\frac{s}{\alpha}\cdot\frac{\alpha q+1}{q^2-1} B=0$, but by Theorem~\ref{theobasisfor3gen}, $BA$ and $B$ are linearly independent.\blitza
\end{proof}
The above corollary includes the case with $s=1$, $\alpha=-1$ and $\beta=1$, which serves as our proof of what was claimed earlier in the introduction that the algebra homomorphism $\uLieAone\into\uLieBone$ in \eqref{notMap1}, that sends $A\mapsto B$ and $B\mapsto-A$, has no counterpart for the $q$-deformed algebras.

\begin{theorem} \label{theoA^WB^yC^kUs}
Given $k, w, y \in \Z^+$ and $t\in\N$,
\begin{eqnarray}
C^tA^yB^wC^k &=&\sum_{i=0}^{y}\binom{y}{i}q^{wy-wi}(\{w\}_qs)^i\sum_{t=0}^{{y-i}}\binom{{y-i}}{t}q^{k(y-i-t)+tw}(\{k\}_qs)^tB^wC^{k+t}A^{{y-i}-t}.  \nonumber
\end{eqnarray}
\end{theorem}

\begin{proof} From \eqref{EqAB^nUs}, we have $AB^w=q^{w}B^wA +\{w\}_qs B^w$. Multiplying both sides by $C^k$ from the right, we find that the equation
  \begin{eqnarray}
A^yB^wC^k&=& \sum_{i=0}^{y}\binom{y}{i}q^{wy-wi}(\{w\}_qs)^iB^wA^{y-i}C^k, \label{EqC^kA^wB^yUs}
  \end{eqnarray}
holds in $\Us$ at $y=1$. Proceeding with induction on $y$, with further use of \eqref{EqAB^nUs}, the proof of \eqref{EqC^kA^wB^yUs} may be completed. From \eqref{EqA^nC^mUs}, we obtain
\begin{eqnarray}
A^{y-i}C^k&=&\sum_{t=0}^{{y-i}}\binom{{y-i}}{t}q^{k(y-i)-kt}(\{k\}_qs)^tC^kA^{{y-i}-t}, \nonumber
\end{eqnarray}
which we substitute into \eqref{EqC^kA^wB^yUs} to obtain
\begin{eqnarray}
A^yB^wC^k &=&\sum_{i=0}^{y}\binom{y}{i}q^{wy-wi}(\{w\}_qs)^i\sum_{t=0}^{{y-i}}\binom{{y-i}}{t}q^{ky-ki-kt}(\{k\}_qs)^tB^wC^kA^{{y-i}-t},  \nonumber
\end{eqnarray}
both sides of which we mutiply by $C^t$ from the right. In the resulting right-hand side, we replace $C^tB^w$ by $q^{wt}B^wC^t$ using \eqref{EqC^mB^nUs}. From this, we get the desired identity.
\end{proof}

We shall relate the Lie algebra $\LiesubUs$ to the corresponding Lie subalgebra $\LiesubUr$ of $\Ur$. We set $r=s$, and so, we shall be dealing with the Lie algebra $\LiesubUss$. A basis for this Lie algebra has been given in Theorem~\ref{TheThm}, which consists of words in the algebra $\Uss$. These words have their counterparts in the algebra $\Us$, which are also shown in \eqref{basisforLUs} below. We define $\scKii$ as the span of the elements
\begin{eqnarray}
    A, B, B^nC^m, C^mA^k,\qquad (m, k \in \Z^+, n \in \N), \label{basisforLUs}
\end{eqnarray}
of $\Us$, and we define $\LieIso:\LiesubUss \longrightarrow \scKii$ to be the linear map  that sends
\begin{eqnarray}
  A&\mapsto& -B, \nonumber \\
  B&\mapsto&-A, \nonumber \\
  C^n&\mapsto&-C^n, \nonumber \\
   C^uA^v&\mapsto&-B^vC^u, \nonumber \\
   B^vC^u&\mapsto&-C^uA^v. \label{missingcase}
\end{eqnarray}

\begin{lemma}\label{LemBCA} For each $k, w, y \in \Z^+$,
\begin{eqnarray}
       \LieIso\lpar B^y C^kA^w\rpar=-B^wC^kA^y. \nonumber
    \end{eqnarray}
\end{lemma}
\begin{proof} The relation \eqref{mainrel5Ur} for the algebra $\Uss$, that is, with the further restriction of $r=s$, is
\begin{eqnarray}
    BC^kA&=&\frac{q^k sC^{k}A -C^{k+1}}{q^{k}(1-q)},\quad k\in \Z^+. \label{mainrel5Uss} 
    \end{eqnarray}
The corresponding relation for the algebra $\Us$ is \eqref{mainrel5Us}. From these two relations, we immediately have $\LieIso\lpar B C^kA\rpar=-BC^kA$. Multiplying both sides of \eqref{mainrel5Uss} by any desired power of $A$ from the right leads to a straightforward computation that proves $\LieIso\lpar B C^kA^w\rpar=-B^wC^kA$, for all $w\in\Z^+$. Suppose that for some $y\in\Z^+$, we have $\LieIso\lpar B^y C^k A^w\rpar=-B^w C^k A^y$ for all $k,w\in\Z^+$. To both sides of \eqref{mainrel5Uss}, we multiply $B^y$ from the left and $A^{w-1}$ from the right to obtain
\begin{eqnarray}
    B^{y+1}C^kA^w&=&\frac{q^k sB^yC^{k}A^w -B^yC^{k+1}A^{w-1}}{q^{k}(1-q)},\quad k\in \Z^+. \label{mainrel5Uss2} 
    \end{eqnarray}
Applying the linear map $\LieIso$ to both sides of \eqref{mainrel5Uss2}, the resulting right-hand side is a linear combination of $\LieIso\lpar B^yC^{k}A^w\rpar$ and $\LieIso\lpar B^yC^{k+1}A^{w-1}\rpar$, both of which  satisfy the inductive hypothesis, except for the case $w=1$ in $\LieIso\lpar B^yC^{k+1}A^{w-1}\rpar$, but for this, we simply use \eqref{missingcase}. Thus,
\begin{eqnarray}
    \LieIso\lpar B^{y+1}C^kA^w\rpar &=&-\frac{q^k sB^wC^{k}A^y -B^{w-1}C^{k+1}A^{y}}{q^{k}(1-q)},\quad k\in \Z^+. \label{mainrel5Uss3} 
    \end{eqnarray}
Multiplying both sides of \eqref{mainrel5Us} by $B^{w-1}$ from the left and by $A^y$ from the right, and combining with \eqref{mainrel5Uss3}, gives us $\LieIso\lpar B^{y+1} C^k A^w\rpar=-B^w C^k A^{y+1}$. By induction, the desired identity follows.
\end{proof}

\begin{lemma}\label{LemCABC} Given $k, w, y \in \Z^+$ and $t\in\N$,
\begin{eqnarray}
\LieIso\lpar C^kA^yB^wC^t\rpar &=&-C^tA^wB^yC^k.  \nonumber
\end{eqnarray}
\end{lemma}
\begin{proof} Using Theorem~\ref{theoC^kA^WB^yUr} on the algebra $\Uss$, or equivalently, setting $r=s$, we have
\begin{eqnarray}
C^kA^wB^yC^t = \sum_{i=0}^{y}\binom{y}{i}q^{wy-wi}(\{w\}_qs)^i\sum_{t=0}^{y-i}\binom{y-i}{t}q^{k(y-i-t)+tw}(\{k\}_qs)^tB^{y-i-t}C^{k+t} A^w,\nonumber  
\end{eqnarray}
on both sides of which, we apply the linear map $\LieIso$. The resulting right-hand side is a linear combination of $\LieIso\lpar B^{y-i-t}C^{k+t} A^w\rpar$, each of which, according to Lemma~\ref{LemBCA}, is equal to $-B^{w}C^{k+t} A^{y-i-t}$. By Theorem~\ref{theoA^WB^yC^kUs}, the entire linear combination is simply $-C^tA^wB^yC^k$.
\end{proof}






\begin{theorem}\label{TheIsoThm} The linear map $\LieIso:\LiesubUss\into\scKii$ is a Lie algebra isomorphism.
\end{theorem}
\begin{proof} An arbitrary linear combination of basis elements of $\LiesubUss$ from \eqref{basisforLUr} may be denoted by
\begin{eqnarray}
\alpha A+\beta B+\sum_n\gamma_nC^n+\sum_{u,v}\mu_{u,v}C^uA^v+\sum_{x,y}\nu_{x,y}B^xA^y,\nonumber
\end{eqnarray}
and its image under $\LieIso$ is 
\begin{eqnarray}
-\alpha B-\beta A-\sum_n\gamma_nC^n-\sum_{u,v}\mu_{u,v}B^vC^u-\sum_{x,y}\nu_{x,y}C^yA^x.\nonumber
\end{eqnarray}
If this linear combination is zero, then the linear independence of \eqref{basisforLUr} implies that all scalar coefficients are zero. Hence, $\ker\LieIso$ is zero. Any basis element of $\scKii$ is the image of some basis element of $\LiesubUss$, so the image of $\LieIso$ is equal to $\scKii$. Hence, $\LieIso$ is a bijective linear map. To prove that $\LieIso$ is a Lie algebra homomorphism, we show that $\LieIso$ preserves the Lie bracket of any two basis elements. To do this, we apply $\LieIso$ to both sides of the commutation relations \eqref{4''}--\eqref{7}. We first consider \eqref{4''}--\eqref{8}:

\begin{eqnarray}
\LieIso\lpar\left[C^u, C^t \right]\rpar&=&0=\lbrak-C^u ,-C^t\rbrak,\label{iso4''} \\
  \LieIso\lpar\left[C^u, A\right]\rpar &=&-(1-q^u)BC^u=\lbrak -C^u,-B\rbrak, \label{iso2'}  \\
\LieIso\lpar\left[ C^uA^w, A\right]\rpar &=&-(1-q^u)B^{w+1}C^u=\lbrak -B^wC^u,-B\rbrak,\label{iso5} \\
\LieIso\lpar  \left[C^uA^w, C^t \right]\rpar&=& -(q^{tw}- 1)B^wC^{t+u}=\lbrak -B^wC^u,-C^t\rbrak, \label{iso7'}  \\
\LieIso\lpar\left[C^uA^w, C^tA^v \right]\rpar&=&-(q^{tw}-q^{uv})B^{v+w}C^{t+u}=\lbrak -B^wC^u,-B^vC^t\rbrak, \label{iso8}
\end{eqnarray}
The second member of each of these equations is a linear combination $F$ of basis elements of $\scKii$, while the third member is a Lie bracket $\lbrak G,H\rbrak$ where $G$ and $H$ are negatives of basis elements of $\scKii$. The equation $F=\lbrak G,H\rbrak$ in each of the above equations (except the first which is trivial) was obtained by routine use of the identity  \eqref{EqC^mB^nUs}. Continuing with the equations \eqref{3'}--\eqref{4}, we have

\begin{eqnarray}
\LieIso\lpar\left[C^u, B \right]\rpar&=&-(q^u-1)C^uA-\{u\}_q rC^{u}=\lbrak -C^u,-A\rbrak, \label{iso3'}\\
\LieIso\lpar\left[B^tC^u, B \right]\rpar&=&-(q^u-1)C^uA^{t+1}-\{u\}_q rC^{u}A^t=\lbrak -C^uA^t,-A\rbrak, \label{iso3}\\
\LieIso\lpar  \left[C^u, B^wC^t \right]\rpar &=&\sum_{i=0}^{w}\binom{w}{i}q^{uw-ui}(\{u\}_qr)^iB^{w-i}C^{t+u}- B^wC^{t+u},\nonumber\\
&=&\lbrak-C^u,-C^tA^w\rbrak, \label{iso4'} \\
  \LieIso\lpar \left[B^vC^u, B^wC^t \right]\rpar  &=&\sum_{i=0}^{w}\binom{w}{i}q^{uw-ui}(\{u\}_qr)^iC^{t+u}A^{v+w-i} \nonumber \\
    &&-\sum_{i=0}^{v}\binom{v}{i}q^{tv-ti}(\{t\}_qr)^iC^{t+u}A^{w+v-i},\nonumber\\
&=&\lbrak-C^uA^v,-C^tA^w\rbrak,   \label{iso4} 
\end{eqnarray}
where the second and third members of the form $F=\lbrak G,H\rbrak$ are routine consequences of \eqref{EqA^nC^mUs}. For the last three equations \eqref{2}--\eqref{7}, we use Lemmas~\ref{LemBCA} and \ref{LemCABC} to obtain
\begin{eqnarray}
 \LieIso\lpar\left[ B^tC^u, A\right]\rpar &=&\LieIso\lpar B^tC^uA-\sum_{i=0}^{t}\binom{t}{i}q^{t+u-i}r^{i}B^{t-i}C^uA\rpar, \nonumber \\  
&=&-BC^uA^t+\sum_{i=0}^{t}\binom{t}{i}q^{t+u-i}r^{i}BC^uA^{t-i}, \nonumber \\  
&=&\sum_{i=0}^{t}\binom{t}{i}q^{t+u-i}r^{i}BC^uA^{t-i}-BC^uA^t, \nonumber \\  
&=&\left[-C^uA^t, -B\right], \label{iso2}\\
 \LieIso\lpar\left[ C^uA^w, B\right]\rpar&=&\LieIso\lpar C^uA^wB-BC^uA^w\rpar, \nonumber\\ 
 &=&-AB^wC^u+B^wC^uA, \nonumber\\ 
  &=&\left[ -B^wC^u, -A\right] ,\label{iso6}\\
\LieIso\lpar\left[C^uA^w, B^yC^t \right]\rpar&=& \LieIso\lpar C^uA^wB^yC^t- B^yC^{t+u}A^w\rpar, \nonumber \\
&=& -C^tA^yB^wC^u+ B^wC^{t+u}A^y, \nonumber \\
&=&\left[ -B^wC^{u}, -C^{t}A^y\right]. \label{iso7}  
\end{eqnarray}
Using the definition of $\LieIso$, the last part or member of each of \eqref{4''}--\eqref{iso7} may be written as a Lie bracket of the form $\lbrak\LieIso(X),\LieIso(Y)\rbrak$ where $X$ and $Y$ are basis elements of $\LiesubUss$. By further inspection, \eqref{iso4''}--\eqref{iso7} imply that for any basis elements $X$ and $Y$ of $\LiesubUss$, we have $\LieIso\lpar\lbrak X,Y\rbrak\rpar=\lbrak\LieIso(X),\LieIso(Y)\rbrak$. This completes the proof.
\end{proof}

\begin{corollary}\label{theCor} The Lie algebras $\LiesubUss$ and $\LiesubUs$ are isomorphic.
\end{corollary}
\begin{proof} In view of Theorem~\ref{TheIsoThm}, showing that $\scKii$ is isomorphic to $\LiesubUs$ shall suffice. By definition, $\scKii$ contains $A,B\in\Us$. Also from Theorem~\ref{TheIsoThm}, $\scKii$ is equal to the image of a Lie algebra homomorphism so it is a Lie subalgebra of $\Us$. What remains to be shown is $\scKii\sub\LiesubUs$. Using Proposition~\ref{nestAdProp} to Lemma~\ref{lemmaadA(CmA&BCmUr}, each word in the basis
\begin{eqnarray}
    A, B, B^nC^m, C^mA^k\in\Uss,\qquad (m, k \in \Z^+, n \in \N), \nonumber
\end{eqnarray}
 for $\LiesubUss$ may be expressed as a Lie polynomial in $A,B\in\Uss$. By Theorem~\ref{TheIsoThm}, $\LieIso$ is a Lie algebra homomorphism, so each of the corresponding words 
\begin{eqnarray}
    A, B, B^nC^m, C^mA^k\in\Us,\qquad (m, k \in \Z^+, n \in \N), \nonumber
\end{eqnarray}
that form a basis for $\scKii$ may be expressed as a Lie polynomial in $\LieIso(A)$ and $\LieIso(B)$. Since $\LieIso(A)=-B$ and $\LieIso(B)=-A$ are both in $\LiesubUs$, we have $\scKii\sub\LiesubUs$. This completes the proof.
\end{proof}

In conclusion, by comparison of Corollaries~\ref{noIsoCor} and \ref{theCor}, we find here an insight on one key difference between the algebra structure, generated by $A$ and $B$, versus the Lie algebra structure (associative structure versus nonassociative structure), under the $q$-deformed commutation relations that we have considered. The traditional algebra homomorphisms that perform $A\mapsto\beta B$ and $B\mapsto \alpha A$ (for some $\alpha,\beta\in\F$), or algebra homomorphisms that ``switch and scale'' the two generators, have no $q$-analogs for the associative structures, but there is one Lie algebra homomorphism of this kind that relates the Lie algebras $\LiesubUss$ and $\LiesubUs$. Furthermore, these two Lie algebra structures are isomorphic, whereas the corresponding associative structures are not. From the standpoint of this study, this kind of insight that contrasts the associative and nonassociative algebraic structures in the same space is one important result that comes from studying Lie polynomial characterization problems.

\end{document}